\documentclass[11pt, a4paper, reqno, english]{amsart}

\usepackage{amssymb,amscd}
\usepackage{hyperref}
\usepackage[all]{xypic}
\usepackage{enumerate}
\usepackage[mathscr]{eucal}
\DeclareFontFamily{U}{wncy}{}
\DeclareFontShape{U}{wncy}{m}{n}{
<5>wncyr5
<6>wncyr6
<7>wncyr7
<8>wncyr8
<9>wncyr9
<10>wncyr10
<11>wncyr10
<12>wncyr6
<14>wncyr7
<17>wncyr8
<20>wncyr10
<25>wncyr10}{}
\DeclareMathAlphabet{\cyr}{U}{wncy}{m}{n}

\newcommand\kbar{{\overline{k}}}

\newcommand\Xbar{{\overline{X}}}
\newcommand\Sbar{{\overline{S}}}
\newcommand\Zbar{{\overline{Z}}}
\newcommand\Ubar{{\overline{U}}}
\newcommand\Ybar{{\overline{Y}}}
\newcommand\Tcbar{{\overline{T^c}}}
\newcommand\kappabar{{\overline{\kappa}}}
\newcommand\That{{\widehat{T}}}
\newcommand\Mhat{{\widehat{M}}}
\newcommand\dhat{{\widehat{d}}}
\renewcommand\tt{\mathbf{t}}
\newcommand\ww{\mathbf{w}}
\newcommand\xx{\mathbf{x}}
\newcommand\yy{\mathbf{y}}
\newcommand\zz{\mathbf{z}}
\renewcommand\AA{\mathbb{A}}
\newcommand\ZZ{\mathbb{Z}}
\newcommand\QQ{\mathbb{Q}}
\newcommand\GG{\mathbb{G}}
\newcommand\GmZ{\GG_{\mathrm{m}}}
\newcommand\Gm[1]{\GG_{\mathrm{m},#1}}
\newcommand\sm{\mathrm{sm}}
\newcommand\T{\mathcal{T}}
\newcommand\gc{\mathrm{cl}}
\DeclareMathOperator\Pic{Pic}
\DeclareMathOperator\Div{Div}
\DeclareMathOperator\divi{div}
\DeclareMathOperator\Gal{Gal}
\DeclareMathOperator\Hom{Hom}
\DeclareMathOperator\Br{Br}

\DeclareMathOperator\inv{inv}
\DeclareMathOperator\Cor{Cor}
\DeclareMathOperator\Res{Res}
\DeclareMathOperator\ord{ord}

\newtheorem{theorem}{Theorem}
\newtheorem{lemma}{Lemma}
\newtheorem{prop}{Proposition}
\newtheorem{cor}{Corollary}
\theoremstyle{definition}
\newtheorem{example}{Example}
\newtheorem*{ack}{Acknowledgements}
\newtheorem*{terminology}{Terminology}
\newtheorem{remark}{Remark}
\numberwithin{equation}{section}

\begin{document}

\title[Quadratic polynomials represented by norms]
{Universal torsors and values of quadratic polynomials represented by
  norms}

\author{Ulrich Derenthal}

\address{Institut f\"ur Algebra, Zahlentheorie und Diskrete Mathematik,
  Leibniz Universit\"at Hannover, Welfengarten 1, 30167 Hannover, Germany}

\email{derenthal@math.uni-hannover.de}

\author{Arne Smeets}

\address{Departement Wiskunde, KU Leuven, Celestijnenlaan 200B, 3001
  Leuven, Belgium \emph{and} D\'epartement de Math\'ematiques,
  B\^atiment 425, Universit\'e Paris-Sud 11, 91405 Orsay, France}

 \email{arnesmeets@gmail.com}

\author{Dasheng Wei}

\address{Academy of Mathematics and System Science, CAS, Beijing 100190,
  P.\ R.\ China}

\email{dshwei@amss.ac.cn}

\date{June 10, 2014}

\begin{abstract}
  Let $K/k$ be an extension of number fields, and let $P(t)$ be a
  quadratic polynomial over $k$. Let $X$ be the affine variety defined
  by $P(t) = N_{K/k}(\zz)$. We study the Hasse principle and weak
  approximation for $X$ in three cases. For $[K:k]=4$ and $P(t)$
  irreducible over $k$ and split in $K$, we prove the Hasse principle
  and weak approximation. For $k=\QQ$ with arbitrary $K$, we show that
  the Brauer-Manin obstruction to the Hasse principle and weak
  approximation is the only one. For $[K:k]=4$ and $P(t)$ irreducible
  over $k$, we determine the Brauer group of smooth proper models of
  $X$. In a case where it is non-trivial, we exhibit a counterexample
  to weak approximation.
\end{abstract}

\subjclass[2010]{14G05 (11D57, 14F22)}

% 14G05: Algebraic geometry->Arithmetic problems. Diophantine geometry->Rational points
%
% 11D57: Number theory->Diophantine equations->Multiplicative and norm form equations
%
% 14F22: Algebraic geometry->(Co)homology theory->Brauer groups of schemes

\maketitle

\tableofcontents

\section{Introduction}

Let $K/k$ be an extension of number fields of degree $n$. When can values of a
polynomial $P(t)$ over $k$ be represented by norms of elements of $K$?  To
answer this natural question, we study solutions $(t,\zz) \in k \times K$ of the
equation
\begin{equation}\label{eq:variety}
  P(t)=N_{K/k}(\zz).
\end{equation}

This question is closely related to the study of the Hasse principle
and weak approximation (see the end of this introduction for a review
of this terminology) on a smooth proper model $X^c$ of the affine
hypersurface $X \subset \AA^1_k \times \AA^n_k$ with coordinates
$(t,\zz)=(t,z_1, \dots, z_n)$ defined by~(\ref{eq:variety}), via a
choice of a basis $\omega_1, \dots, \omega_n$ of $K$ over $k$, with
$N_{K/k}(\zz) = N_{K/k}(z_1\omega_1+\dots+z_n\omega_n)$.

Colliot-Th\'el\`ene conjectured that the Brauer--Manin obstruction to
weak approximation is the only one on $X^c$ (see \cite{MR2011747}).
This conjecture is known in the case where $P(t)$ is constant, thanks
to work of Sansuc \cite{MR631309}; if additionally $K/k$ is cyclic, it
is known that the Hasse principle (proved by Hasse himself
\cite[p.~150]{56.0165.03}) and weak approximation hold.  Other known
cases of Colliot-Th\'el\`ene's conjecture, in some cases leading to a
proof of the Hasse principle and weak approximation, include the class
of Ch\^atelet surfaces ($[K:k]=2$ and $\deg(P(t)) \le 4$)
\cite{MR870307}, \cite{MR876222}, a class of singular cubic
hypersurfaces ($[K:k]=3$ and $\deg(P(t)) \le 3$) \cite{MR988101} and
the case where $K/k$ is arbitrary and $P(t)$ is split over $k$ with at
most two distinct roots \cite{MR1961196}, \cite{MR2053456},
\cite{arXiv:1111.4089}. Finally, if one admits Schinzel's hypothesis,
then the conjecture is known for $K/k$ cyclic and $P(t)$ arbitrary
\cite{MR1285781}, \cite{MR1603908}. See for example \cite[Introduction]{MR2011747} and
\cite[Section~1]{arXiv:1109.0232} for a more detailed discussion of
these results and the difficulties of this problem.

\medskip

The obvious next challenge is the case where $P(t)$ is an irreducible
quadratic polynomial. Browning and Heath-Brown recently obtained a positive
result in this direction: they proved the Hasse principle and weak
approximation for $[K:k]=4$ and $\deg(P(t))=2$ with $P(t)$ irreducible over
$k$ and split in $K$, with $k=\QQ$. Their main result
\cite[Theorem~1]{arXiv:1109.0232} therefore answers a question raised in
\cite[Section~2]{MR2053456} positively in the case $k =
\QQ$. In this paper, we extend their work in several directions, using a
variety of different techniques.

We give a very short proof of this result for an arbitrary number
field~$k$.  It is independent of the work of Browning and Heath-Brown
and gives a simple geometric proof of their result. More
specifically, we show that an open subset of $X$ is a smooth fibration
in $3$-dimensional quadrics defined by quadratic forms of full rank
over a conic. The result follows by applying the theorem of
Hasse--Minkowski to the base and to the fiber.

\begin{theorem}\label{thm:degree_4}
  Let $P(t)$ be a quadratic polynomial that is irreducible over a number field
  $k$ and split in $K$ with $[K:k]=4$. Then the Hasse principle and weak
  approximation hold for the variety $X \subset \AA^5_k$ defined
  by~(\ref{eq:variety}).
\end{theorem}

If the ground field is $\QQ$, we can prove a much more general result
based on the analytic work of Browning and Heath-Brown in
\cite[Theorem~2]{arXiv:1109.0232} and the descent method of
Colliot-Th\'el\`ene and Sansuc:

\begin{theorem}\label{thm:Q}
  Let $k = \QQ$ and $K$ be any number field. Let $P(t) \in \QQ[t]$ be
  an arbitrary quadratic polynomial. Then the Brauer--Manin
  obstruction to the Hasse principle and weak approximation is the
  only obstruction on any smooth proper model of $X \subset
  \AA^{n+1}_\QQ$ defined by~(\ref{eq:variety}).
\end{theorem}

Let $X$ be the variety defined by equation~(\ref{eq:variety}) and let $U
\subset X$ be the open subvariety given by $P(t) \neq 0$. We prove that the
variety $Y$ defined by \cite[(1.5)]{arXiv:1109.0232} is the restriction $\T_U$
of a universal torsor $\T$ over $X$ to $U$, or a product of $\T_U$ with a
quasi-split torus. For the variety $Y$, \cite[Theorem~2]{arXiv:1109.0232}
proves weak approximation using sieve methods from analytic number theory,
inspired by work of Fouvry and Iwaniec \cite{MR1438827}.  While one step in
Browning's and Heath-Brown's deduction of \cite[Theorem~1]{arXiv:1109.0232}
from \cite[Theorem~2]{arXiv:1109.0232} leads to their restriction to
$[K:\QQ]=4$, the combination of their analytic work with descent theory
gives our more general Theorem~\ref{thm:Q}.
We also generalize Theorem~\ref{thm:Q} to a large class of multivariate
polynomials $P(t_1,\dots,t_\ell) \in \QQ[t_1,\dots,t_\ell]$.

Note that the descent argument in our proof of Theorem~\ref{thm:Q} follows the
proof of \cite[Theorem~2.2]{MR1961196}. As pointed out by the referee, the
more sophisticated method of \cite{MR2053456} would allow the treatment of
smooth proper models of $X \subset \AA^{n+1}_\QQ$ defined by the more general
equation
\begin{equation*}
  cP(t)^n=N_{K/\QQ}(\zz)
\end{equation*}
for any $n \ge 1$, with $c \in \QQ^\times$, quadratic $P(t) \in \QQ[t]$ and
any number field $K$. See also \cite{arXiv:1210.5727}.

Our Theorem~\ref{thm:Q} then leads naturally to the question in which
cases there may be a Brauer--Manin obstruction.  We give a complete
analysis for $[K:k]=4$ and $P(t)$ irreducible over $k=\QQ$. If $P(t)$
is split in $K$, there is no Brauer--Manin obstruction to weak
approximation by \cite[Theorem~1]{arXiv:1109.0232} or our
Theorem~\ref{thm:degree_4}, so that it is not very surprising that the
Brauer group of a smooth proper model of $X$ is trivial in this case.
Otherwise, the Brauer group is sometimes non-trivial:

\begin{theorem}\label{thm:brauer_short}
  Let $P(t)$ be a quadratic polynomial that is irreducible over a
  number field $k$ and has splitting field $L$, and let $[K:k]=4$. Let
  $X^c$ be a smooth proper model of $X \subset \AA^5_k$ defined
  by~(\ref{eq:variety}).

  The Brauer group $\Br(X^c)$ contains non-constant elements if and
  only if the extension $K/k$ is not Galois, $P(t)$ is not split in
  $K$ and the extension $K\cdot L/k$ is Galois with $\Gal(K\cdot L/L)
  \cong \ZZ/4\ZZ$.
\end{theorem}

In Section~\ref{sec:brauer}, we prove a more detailed version of
Theorem~\ref{thm:brauer_short} that includes an explicit description
of the Brauer group in the non-trivial case
(Theorem~\ref{thm:brauer}). The calculations are quite involved and it
seems unlikely that similar detailed results can be obtained for
higher degree extensions.

As a corollary, we get the validity of the Hasse principle and weak
approximation in some new cases when the Brauer group is trivial
(Corollary~\ref{cor:quartic}).  In Section~\ref{sec:counterexample},
we use our explicit description of the Brauer group in the case where
it is non-trivial to produce a counterexample to weak approximation
(explained by the Brauer-Manin obstruction).

\begin{terminology}
  For an algebraic variety $Z$ defined over a number field $k$ with
  algebraic closure $\kbar$, one says that the Hasse principle holds
  if $\prod_{v \in \Omega_k} Z(k_v) \ne \emptyset$ (where $\Omega_k$
  is the set of places of $k$ and $k_v$ is the completion of $k$ at
  $v$) implies $Z(k) \ne \emptyset$. One says that weak approximation
  holds if $Z(k)$ is dense in $\prod_{v \in \Omega_k} Z(k_v)$ with the
  product topology, via the diagonal embedding.

  If $Z$ is smooth and proper, one says that the Brauer--Manin
  obstruction to the Hasse principle is the only one if $(\prod_{v \in
    \Omega_k} Z(k_v))^{\Br(Z)} \ne \emptyset$ implies that $Z(k) \ne
  \emptyset$, and that the Brauer--Manin obstruction to weak
  approximation is the only one if $Z(k)$ is dense in $(\prod_{v \in
    \Omega_k} Z(k_v))^{\Br(Z)}$. Here $(\prod_{v \in \Omega_k}
  Z(k_v))^{\Br(Z)}$ is the set of all $(z_v) \in \prod_{v \in
    \Omega_k} Z(k_v)$ satisfying $\sum_{v \in \Omega_k} \inv_v(A(z_v))
  = 0$ for each $A$ in the Brauer group $\Br(Z) =
  H^2_\text{\'et}(Z,\GmZ)$ of $Z$, where the map $\inv_v : \Br(k_v)
  \to \QQ/\ZZ$ is the invariant map from local class field theory.

  The subgroup $\Br_0(Z)$ of constant elements in the Brauer group is
  the image of the natural map $\Br(k) \to \Br(Z)$. The algebraic
  Brauer group $\Br_1(Z)$ is the kernel of the natural map $\Br(Z) \to
  \Br(\Zbar)$, where $\Zbar = Z \times_k \kbar$.
\end{terminology}

\begin{ack}
  The first named author was supported by grant DE 1646/2--1 of the
  Deutsche Forschungsgemeinschaft and grant 200021\underline{\
    \,}124737/1 of the Schweizer Nationalfonds. The second named
  author was supported by a PhD fellowship of the Research Foundation
  -- Flanders (FWO). The third named author was supported by National
  Key Basic Research Program of China (Grant No. 2013CB834202) and
  National Natural Science Foundation of China (Grant Nos. 11371210 and
  11321101). This collaboration was supported by the Center for
  Advanced Studies of LMU M\"unchen. We thank T.~D.~Browning,
  J.\hbox{-}L.~Colliot-Th\'el\`ene, C.~Demarche and B.~Kunyavski{\u\i}
  for useful discussions and remarks. Finally, we thank the referee
  for his suggestions for improvement.
\end{ack}

\section{Quadratic polynomials represented by a quartic norm}

In this section, we give a very short proof of
Theorem~\ref{thm:degree_4} that is independent of the work of Browning
and Heath-Brown \cite{arXiv:1109.0232} and generalizes it from $\QQ$
to an arbitrary number field $k$.

\begin{proof}[Proof of Theorem~\ref{thm:degree_4}]
  Using a change of variables if necessary, we can assume that the
  irreducible quadratic polynomial $P(t) \in k[t]$ that is split in
  $K$ (with $[K:k]=4$) has the form $P(t)=c(t^2-a)$, with $c \in
  k^\times$, where $a \in k^\times$ is not a square and $\sqrt a \in
  K$. Write $L = k(\sqrt a) \subset K$.

  Let $U = \{(t,\zz) : P(t) \ne 0\} \subset X$. Let $S\subset \AA^2_k$ be the
  conic defined by the affine equation $N_{L/k}(\ww)=c$ and let $p: U \to S$
  be the morphism defined by
  \begin{equation*}
    (t,\zz) \mapsto (t-\sqrt a)^{-1} N_{K/L}(\zz).
  \end{equation*}

  Theorem~\ref{thm:degree_4} holds by \cite[Proposition~3.9]{MR870307},
  which is based on the theorem of Hasse--Minkowski,
  since this is a smooth fibration in quadrics of dimension $3$
  over a base satisfying weak approximation.

  Indeed, the base $S$ is a conic, where weak approximation holds, and
  the statement regarding the fibers can be checked over an algebraic
  closure $\kbar$ of $k$. Here, we have
  \begin{equation*}
    \Ubar = U \times_k \kbar \cong \{c(t^2-a)=u_1u_2u_3u_4\} \subset \AA^5_\kbar, \
    \Sbar = S \times_k \kbar \cong \{w_1w_2=c\} \subset \AA^2_\kbar,
  \end{equation*}
  with $p$ mapping $(t,u_1,\dots,u_4) \in \Ubar$ to $((t-\sqrt
  a)^{-1}u_1u_2, (t+\sqrt a)^{-1}u_3u_4) \in \Sbar$. Hence the fiber
  over $(w_1,w_2) \in \Sbar$ is
  \begin{equation*}
    \{t-\sqrt a=w_1^{-1}u_1u_2,\ t+\sqrt a=w_2^{-1}u_3u_4\}\subset \AA^5_\kbar.
  \end{equation*}
  Eliminating $t$ gives
  \begin{equation*}
    \{-2\sqrt a=w_1^{-1}u_1u_2-w_2^{-1}u_3u_4\}\subset \AA^4_\kbar,
  \end{equation*}
  where the quadratic form in $u_1, \dots, u_4$ clearly has rank $4$.
  Hence the fiber is a smooth quadric of dimension $3$.
\end{proof}

\begin{remark}
  The analog of Theorem~\ref{thm:degree_4} holds for global fields of positive
  characteristic different from $2$ as well. Indeed, it is not hard to see
  that our arguments and the proof of \cite[Proposition~3.9]{MR870307} remain
  valid for such fields.
\end{remark}

\section{Universal torsors}\label{sec:torsor}

The basic strategy to prove Theorem~\ref{thm:Q} is based on the
following result, which reduces the problem of the Hasse principle and
weak approximation on a variety to the same questions on its universal
torsors, where we have no Brauer--Manin obstructions.  This kind of
result has been proved essentially by Colliot-Th\'el\`ene and Sansuc
in their seminal paper \cite{MR89f:11082}. However, they developed
their theory under the simplifying assumption that the varieties
involved are proper. Skorobogatov developed a variant under less
stringent assumptions in \cite{MR1666779}. Descent on open varieties
also features in \cite{MR1744112} and \cite{MR2011747}.  We will use
the following variant:

\begin{prop}\label{lem:torsor}
  Let $Z$ be a smooth, geometrically rational variety over a number
  field $k$ with algebraic closure $\kbar$. Let $\Zbar = Z \times_k
  \kbar$. Assume furthermore that $\overline{k}[Z]^\times =
  \overline{k}^\times$, that $\Pic(\Zbar)$ is free of finite rank,
  that universal $Z$-torsors exist and that any universal $Z$-torsor
  satisfies weak approximation.  Then the Brauer--Manin obstruction to
  the Hasse principle and weak approximation is the only one for any
  smooth proper model $Z^c$ of~$Z$.
\end{prop}

The condition $\kbar[Z]^\times = \kbar^\times$ means that the only invertible
regular functions on $Z$ are the constant ones.
The proof of this proposition is straightforward; the key
observation is the fact that $Z(\AA_k)^{\Br_1(Z)}$ is dense in
$(\prod_{v}Z^c(k_v))^{\Br(Z^c)}$ by
\cite[Proposition~1.1]{MR1744112} (note that $\Br(\overline{Z^c}) =
 0$), and the result then follows from descent theory and the implicit function theorem.

\medskip

The main result of this section is concerned with the existence of
universal torsors \cite[(2.0.4)]{MR89f:11082} over $X$ as
in~(\ref{eq:variety}) and their local description.

Let us first recall some more definitions. If $k$ is a field and if $A$ is an
\'etale $k$-algebra, then the $k$-variety $R_{A/k}(\Gm{A})$ is defined via its
functor of points: take $R_{A/k}(\Gm{A})(B) = (A \otimes_k B)^\times$
functorially for every $k$-algebra $B$. The norm map $N_{A/k}$ is defined as
in \cite[\S 12.2]{MR0098114}. We denote the absolute Galois group of $k$ by
$\Gamma_k$.

\begin{prop}\label{prop:torsor}
  Let $K/k$ be an extension of fields of degree $n$. Let $P(t)$ be an
  irreducible separable polynomial of degree $r$ over $k$.

  The variety $X \subset \AA^{n+1}_k$ defined by~(\ref{eq:variety}) is smooth
  and geometrically integral, with $\Pic(\Xbar)$ free of finite rank and
  $\kbar[X]^\times = \kbar^\times$. Let $U$ be the open subset of $X$ defined
  by $P(t) \ne 0$. Then $\Pic(\Ubar) = 0$.

  Let $c \in k^\times$ be the leading coefficient of $P(t)$, let $L$ be the
  field $k[t]/(P(t))$ and let $\eta$ be the class of $t$ in $L$. Let $A = L
  \otimes_k K$. For any universal torsor $\T$ over $X$, there exists
  a solution $(\rho,\xi) \in L^\times \times K^\times$ of the equation 
  \begin{equation}\label{eq:splitting_condition}
    c N_{L/k}(\rho) = N_{K/k}(\xi)
  \end{equation}
  such that $\T_U$ (its restriction to
  $U$) is isomorphic to the subvariety of $\AA^1_k \times
  R_{A/k}(\Gm{A})$ (with coordinates $(t,\zz)$) given by the
  equation
  \begin{equation}\label{eq:torsor_equation_general}
    t - \eta = \rho N_{A/L}(\zz).
  \end{equation}
  Conversely, for any solution $(\rho,\xi) \in L^\times \times
  K^\times$ of \eqref{eq:splitting_condition}, there is a universal
  torsor $\T$ over $X$ such that $\T_U$ has such a description.
\end{prop}

Using only the basic definitions, it is easy to see that one can specialize
equation (\ref{eq:torsor_equation_general}) as follows in the two ``extreme''
cases:
\begin{itemize}
\item[(a)] If $P(t)$ splits completely in $K$, then $\T_U$ is
  isomorphic to the sub-variety of $\AA^1_k \times (R_{K/k}(\Gm{K}))^r$ (with
  coordinates $(t,\xx_1,\dots,\xx_r)$) given by the equation
    \begin{equation}
      t-\eta = \rho\prod_{i=1}^r \sigma_i^{-1}(N_{K/\sigma_i(L)}(\xx_i)) \label{eq:torsor_equation}
    \end{equation}
    where $\sigma_1, \dots, \sigma_r$ is a set of representatives of
    $\Gamma_k/\Gamma_L$.
  \item[(b)] If $P(t)$ remains irreducible in $K$, then $\T_U$ is
    isomorphic to the subvariety of $\AA^1_k \times R_{F/k}(\Gm{F})$ (with
    coordinates $(t,\xx)$) given by the equation
    \begin{equation}
      t-\eta = \rho N_{F/L}(\xx) \label{eq:torsor_equation_2} \end{equation}
    where $F = L \cdot K$.
\end{itemize}

\medskip

The proof of Proposition~\ref{prop:torsor} is an adaptation of
\cite[Theorem~2.2]{MR1961196} and will occupy most of the remainder of
this section.  The $\kbar$-variety $\Xbar$ can be described by an
equation of the form
\begin{equation}\label{eq:variety_kbar}
  c\prod_{i=1}^r(t-\eta_i)=u_1\cdots u_n
\end{equation}
where $\eta_1, \dots, \eta_r$ are the embeddings of $\eta$ in
$\kbar$. We note that $X$ is smooth because $P(t)$ is
separable. Consider the morphism $p: X \to \AA^1_k$ given by $(t,\xx)
\mapsto t$. Over $\kbar$, it has precisely $r$ reducible fibers $X_i$,
for $i=1, \dots, r$, over $t = \eta_i$. Each of these has $n$
irreducible components $D_{i,j} = \{t=\eta_i,\ u_j=0\}$ for $j=1,
\dots, n$. Let $U_0$ be the open subset of $\AA^1_k$ where $P(t) \ne
0$ and let $U = p^{-1}(U_0) \subset X$. We have
\begin{equation*}
  \Ubar = U \times_k \kbar \cong
  (\AA^1_\kbar \setminus\{\eta_1, \dots, \eta_r\}) \times \Gm{\kbar}^{n-1},
\end{equation*}
so that $\Pic(\Ubar) = 0$.

We have $\kbar[X]^\times = \kbar^\times$.  Indeed, the generic fiber of $\Xbar
\to \AA^1_\kbar$ is $\Gm{\kbar(t)}^{n-1}$. Therefore, any $f \in
\kbar[X]^\times$ has the form $f=g(t)u_1^{m_1}\cdots u_n^{m_n}$ with $g \in
k(t)$ and $m_1, \dots, m_n \in \ZZ$. If $g(t)$ has a root or pole in some $t_0
\notin \{\eta_1, \dots, \eta_r\}$, then $f$ or $f^{-1}$ is not regular in a
point on $p^{-1}(t_0)$. Otherwise, we have
\begin{equation*}
g(t)=c'\prod_{i=1}^r(t-\eta_i)^{e_i}
\end{equation*}
for some $c' \in \kbar^\times$
and $e_1, \dots, e_r \in \ZZ$. Then
\begin{equation*}
  \divi(f)=\sum_{i=1}^r \sum_{j=1}^n (e_i + m_j)D_{i,j},
\end{equation*}
so $f \in \kbar[X]^\times$ if and only if
$e_1=\dots=e_r=-m_1=\dots=-m_n$. By~(\ref{eq:variety_kbar}), this is
equivalent to saying that $f$ is a constant in $\kbar^\times$.

By descent theory \cite[Corollary~2.3.4]{MR89f:11082}, universal torsors over
$X$ exist if and only if the exact sequence of $\Gamma_k$-modules
\begin{equation}\label{eq:exact}
  1 \to \kbar^\times \to \kbar[U]^\times \to \kbar[U]^\times/\kbar^\times \to 1
\end{equation}
is split.

It is easy to see that the abelian group $\kbar[U]^\times/\kbar^\times$ is
free of rank $r+n-1$, generated by the classes of the functions $t-\eta_1,
\dots, t-\eta_r, u_1, \dots, u_n$ with an obvious $\Gamma_k$-action and the
relation
\begin{equation}\label{eq:relation}
  \sum_{i=1}^r [t-\eta_i] - \sum_{j=1}^n [u_j] = 0
\end{equation}
because of the equation defining $X$.

The exact sequence~(\ref{eq:exact}) is split if and only if the
classes can be lifted to $\kbar[U]^\times$ in a $\Gamma_k$-equivariant
way, via a map
\begin{equation}\label{eq:splitting}
  \phi: \kbar[U]^\times/\kbar^\times \to \kbar[U]^\times, \quad
  [t-\eta] \mapsto \rho^{-1}(t-\eta),\quad [u_1] \mapsto \xi^{-1} u_1
\end{equation}
where $\rho \in L^\times$ and $\xi \in K^\times$. Because of the
unique relation~\eqref{eq:relation}, the pair $(\rho,\xi) \in L^\times \times
K^\times$ defines such a splitting if and only if
it satisfies \eqref{eq:splitting_condition}.

We now want to apply \cite[Theorem~2.3.1, Corollary~2.3.4]{MR89f:11082} for
the local description of universal torsors over $X$. We will describe a
morphism of tori $d: M \to T$ such that its dual map of characters fits into
the following commutative diagram of $\Gamma_k$-equivariant homomorphisms.
\begin{equation*}
  \begin{CD}
    0 @>>> \That @>\dhat>> \Mhat @>>> \Pic(\Xbar) @>>> 0\\
    @. @V\sim V i V @V\sim V j V @| @.\\
    1 @>>> \kbar[U]^\times/\kbar^\times @>\divi{} >> \Div_{\Xbar\setminus
      \Ubar}(\Xbar) @>>> \Pic(\Xbar) @>>> 0
  \end{CD}
\end{equation*}
Here, the second row is exact because $\Pic(\Ubar)=0$ and $\kbar[X]^\times =
\kbar^\times$.

The $\Gamma_k$-module $\kbar[U]^\times/\kbar^\times$ is isomorphic to the
module of characters of the algebraic $k$-torus $T \subset R_{L/k}(\Gm{L})
\times R_{K/k}(\Gm{K})$ with coordinates $(\zz_1,\zz_2$) given by
\begin{equation*}
  N_{L/k}(\zz_1) = N_{K/k}(\zz_2).
\end{equation*}
Indeed, the character group $\That$ is the quotient of $\ZZ[\Gamma_k/\Gamma_L]
\oplus \ZZ[\Gamma_k/\Gamma_K]$ with the diagonal $\Gamma_k$-action by the
relation
\begin{equation*}
  \sum_{\sigma \Gamma_L \in \Gamma_k/\Gamma_L} \sigma \Gamma_L =
  \sum_{\gamma \Gamma_K \in \Gamma_k/\Gamma_K} \gamma \Gamma_K.
\end{equation*}
The isomorphism $i: \That \to \kbar[U]^\times/\kbar^\times$ is given by
\begin{equation*}
  i(\sigma \Gamma_L) = [t - \sigma(\eta)],\quad i(\gamma \Gamma_K) = [\gamma(u_1)].
\end{equation*}

The abelian group $\Div_{\Xbar \setminus \Ubar}(\Xbar)$ is free of rank $rn$,
generated by $D_{i,j}$ for $i=1, \dots, r$ and $j=1, \dots, n$.  There is a
bijection $\Gamma_k/\Gamma_L \times \Gamma_k/\Gamma_K \to \{D_{i,j}\}$ defined
by $(\sigma\Gamma_L,\gamma\Gamma_K) \mapsto \{t=\sigma(\eta),\gamma(u_1)=0\}$
that is compatible with the action of $\Gamma_k$, acting diagonally on the
left hand side. Recalling $A = L \otimes_k K$, this shows that $\Div_{\Xbar
  \setminus \Ubar}(\Xbar)$ is isomorphic to the module of characters of the
$k$-torus $M = R_{A/k}(\Gm{A})$. Let $j: \Mhat \to \Div_{\Xbar \setminus
  \Ubar}(\Xbar)$ be this isomorphism.

Consider the homomorphism $\divi :\kbar[U]^\times/\kbar^\times \to
\Div_{\Xbar \setminus \Ubar}(\Xbar)$ that maps a function to its divisor. We
have
\begin{equation*}
  \divi([t-\eta]) = \sum_{j=1}^n D_{1,j}, \quad \divi([u_1]) = \sum_{i=1}^r D_{i,1}.
\end{equation*}
Therefore, the abelian group $\Pic(\Xbar)$ is free of rank $(r-1)(n-1)$, with a basis consisting of the classes
$[D_{i,j}]$ for $i=1, \dots, r-1$ and $j=1, \dots, n-1$.  
Now $\divi$ induces a homomorphism on the character modules $\dhat:
\That \to \Mhat$. The dual of this homomorphism is then given by the
morphism of $k$-tori
\begin{equation*}
  d: M \to T,\quad \zz \mapsto (N_{A/L}(\zz),N_{A/K}(\zz)).
\end{equation*}

Let $S$ be the N\'eron--Severi torus dual to the $\Gamma_k$-module
$\Pic(\Xbar)$, so that we have an exact sequence of tori
\begin{equation*}
  1 \to S \to M \to T \to 1.
\end{equation*}
This makes $M$ into a $T$-torsor under $S$.

We now describe the map $U \to T$ induced by the splitting $\phi$ as
in~(\ref{eq:splitting}) by a choice of $(\rho,\xi) \in L^\times \times
K^\times$ satisfying~(\ref{eq:splitting_condition}). The induced map is given
by
\begin{equation*}
  U \to T,\quad (t,\xx) \mapsto (\rho^{-1}(t-\eta),\xi^{-1}\xx),
\end{equation*}
and it is easy to see that the image is in $T$ using the equation of $X$ and
the condition~(\ref{eq:splitting_condition}). Therefore, the image of $U$ in
$T$ is isomorphic to the subvariety of $\AA^1_k \times T$ with coordinates
$(t,\zz_1,\zz_2)$ defined by
\begin{equation*}
  t-\eta = \rho \zz_1.
\end{equation*}

By \cite[Theorem~2.3.1, Corollary~2.3.4]{MR89f:11082}, any universal torsor
$\T_U$ over $U$ is the pullback of a torsor $M$ from $T$ to $U$. Our
computations show that it is isomorphic to the subvariety of $\AA^1_k \times
R_{A/k}(\Gm{A})$ with coordinates $(t,\zz)$ defined
by~(\ref{eq:torsor_equation}).  This completes the proof of Proposition
\ref{prop:torsor}. \hfill $\square$

\begin{remark}\label{rem:reducible}
  One can determine equations for universal torsors $\T$ over the smooth locus
  $X_\sm$ of the variety $X$ defined by~(\ref{eq:variety}) even if $P(t)$ is
  not irreducible over $k$; note that $X$ is not smooth if $P(t)$ is not
  separable. Then $\Pic(X_\sm)$ is a finitely generated (but not necessarily
  free) abelian group. So $\T$ will be a torsor over $X_\sm$ under the group
  of multiplicative type that is dual to $\Pic(X_\sm)$.

  The result is as follows: Assume that
  \begin{equation*}
    P(t) = cP_1(t)^{e_1}\cdots P_d(t)^{e_d}
  \end{equation*}
  for $c \in k^\times$, some irreducible monic polynomials $P_i(t) \in k[t]$
  and positive integers $e_i$. Write $L_i = k[t]/(P_i(t))$ and let $\eta_i$ be
  the class of $t$ in $L_i$. For $i=1, \dots, d$, consider the \'etale
  $L_i$-algebra $A_i = L_i \otimes_k K$. Let $U \subset X_\sm$ be the open
  subvariety given by $P(t) \neq 0$. For any universal torsor $\T$ over
  $X_\sm$, there exists a solution $(\rho_1,\dots,\rho_d,\xi) \in L_1^\times
  \times \cdots \times L_d^\times \times K^\times$ of the equation
  \begin{equation}\label{eq:ut_nonsep}
    cN_{L_1/k}(\rho_1)^{e_1} \cdots N_{L_d/k}(\rho_d)^{e_d} = N_{K/k}(\xi)
  \end{equation}
  such that $\T_U$ is isomorphic to the subvariety of $\AA^1_k
  \times \prod_{i = 1}^d R_{A_i/k}(\Gm{A_i})$ with coordinates
  $(t,\zz_1,\dots,\zz_d)$ given by the system of equations
  \begin{equation*}
    t - \eta_i = \rho_i N_{A_i/L_i}(\zz_i)\ \text{ for }\ 1 \leq i \leq d.
  \end{equation*}
  Conversely, for any solution of \eqref{eq:ut_nonsep}, there is a
  universal torsor $\T$ over $X_\sm$ with such a local description.

  Note that \cite[Theorem~2.2]{MR1961196} is a special case of this
  result. The proof is an adaptation of the proof of
  \cite[Theorem~2.2]{MR1961196} or Proposition~\ref{prop:torsor}.
\end{remark}

\medskip

In case that $k$ is a number field, we can link the existence of
universal torsors as in Proposition~\ref{prop:torsor} to the absence
of Brauer--Manin obstructions on $X$. We have the following general
statement, suggested to us by the referee:

\begin{lemma}\label{lemm:torsorbis}
  Let $X$ be a smooth, geometrically integral variety over a number
  field $k$ such that $\kbar[X]^\times = \kbar^\times$.  Assume that
  $\Pic(\Xbar)$ is a finitely generated abelian group.  If there is no
  Brauer--Manin obstruction to the Hasse principle on a smooth proper
  model of $X$, then universal torsors over $X$ exist.
\end{lemma}

\begin{proof}
  Since $\Pic(\Xbar)$ is finitely generated, there exists an open
  subset $U \subset X$ such that $\Pic(\Ubar) = 0$.  By
  \cite[Proposition~2.2.8]{MR89f:11082}, universal $X$-torsors exist
  if and only if the exact sequence
  \begin{equation*}
    1 \to \kbar^\times \to \kbar[U]^\times \to
    \kbar[U]^\times/\kbar^\times \to 1
  \end{equation*}
  of Galois modules is split. The same result applies, of course, to
  any smooth proper model $X^c$ of $X$.  Hence universal $X$-torsors
  exist if and only if universal $X^c$-torsors exist; by
  \cite[Proposition 6.1.4]{MR1845760}, this is the case exactly when
  $X^c(\AA_k)^{\Br(X^c)} \neq \emptyset$.
\end{proof}

\section{Quadratic polynomials represented by a norm over $\QQ$}

Let $k = \QQ$. As before, we can assume without loss of generality that
$P(t)=c(t^2-a)$ with $c \in \QQ^\times$ and $a \in \QQ$, but now we do not
assume that $P(t)$ is split in $K$. Using the deep work of Browning and
Heath-Brown and our description of universal torsors, we can prove the
following result:

\begin{prop}\label{lem:WA_over_Q}
  If the quadratic polynomial $P(t)$ is irreducible over $\QQ$, then
  each universal torsor over $X \subset \AA^{n+1}_k$ defined
  by~(\ref{eq:variety}) satisfies weak approximation.
\end{prop}

\begin{proof}
  Let $\T$ be a universal torsor over $X$, with $\T_U$ as in
  Proposition~\ref{prop:torsor}. It is enough to prove weak approximation on
  $\T_U$.

  Assume that $P(t)$ is split in $K$.
  Consider $\T_U\subset \AA^1_k \times (R_{K/k}(\Gm{K}))^2$ defined by
  equation~(\ref{eq:torsor_equation}) in the case $r=2$. Let $L=k(\sqrt
  a)$. For any $\sigma \in \Gamma_k$, we have $\sigma(L)=L$, and for any $x
  \in L$, we have $\sigma(x)=\sigma^{-1}(x)$.  Therefore,
  (\ref{eq:torsor_equation}) can be rewritten as
  \begin{equation*}
    t-\sqrt a = \rho N_{K/L}(\xx_1)\cdot \sigma(N_{K/L}(\xx_2)),
  \end{equation*}
  where $\sigma \in \Gamma_k$ with $\sigma(\sqrt a)=-\sqrt a$.

  The variety given by this equation is isomorphic to the subvariety
  $Y$ of $\AA^1_k \times (R_{K/k}(\Gm{K}))^2$ defined by the equation
  \begin{equation}\label{eq:torsor_BHB} N_{K/k}(\ww)(t-\sqrt a) =
    \rho N_{K/L}(\yy),
  \end{equation} via the substitution
  \begin{equation*}
    \ww = \xx_2^{-1},\quad \yy = \xx_1\xx_2^{-1}
  \end{equation*}
  with inverse
  \begin{equation*}
    \xx_1=\ww^{-1}\yy,\quad \xx_2=\ww^{-1}
  \end{equation*}
  using $N_{K/k}(\xx_2) = N_{L/k}(N_{K/L}(\xx_2)) = N_{K/L}(\xx_2) \cdot
  \sigma(N_{K/L}(\xx_2))$. This is exactly
  \cite[equation~(1.5)]{arXiv:1109.0232}. Weak approximation then holds on $Y$
  because of \cite[Theorem~2]{arXiv:1109.0232}.

  Assume now that $P(t)$ remains irreducible over $K$ and write $F = K \cdot
  L$, where $L = k(\sqrt{a})$. Choose some $\sigma \in \Gamma_K$ such that
  $\sigma \notin \Gamma_F=\Gamma_L \cap \Gamma_K$, so $\sigma \notin
  \Gamma_L$. Therefore, $\sigma$ is a representative of the non-trivial class
  both in $\Gamma_K/\Gamma_F$ and in $\Gamma_k/\Gamma_L$.

  Let $\gamma_1, \dots, \gamma_n$ be a set of coset representatives of
  $\Gamma_L/\Gamma_F$. We claim that a set of representatives of
  $\Gamma_k/\Gamma_F$ is given by $\gamma_1, \dots, \gamma_n,
  \gamma_1\sigma, \dots, \gamma_n\sigma$. Indeed, if
  $\gamma_i\sigma\Gamma_F = \gamma_j\sigma\Gamma_F$, then we have
  $\sigma^{-1}\gamma_j^{-1}\gamma_i\sigma \in \Gamma_F = \Gamma_L \cap
  \Gamma_K$.  Using $\sigma \in \Gamma_K$, this gives
  $\gamma_j^{-1}\gamma_i \in \sigma\Gamma_K\sigma^{-1} = \Gamma_K$,
  and we have $\gamma_j^{-1}\gamma_i \in \Gamma_L$ by definition.
  Hence $\gamma_j^{-1}\gamma_i \in \Gamma_L \cap \Gamma_K = \Gamma_F$,
  so $\gamma_i\Gamma_F = \gamma_j\Gamma_F$, which implies
  $i=j$. Furthermore, if $\gamma_i\sigma\Gamma_F = \gamma_j\Gamma_F$,
  then $\gamma_j^{-1}\gamma_i\sigma \in \Gamma_F \subset \Gamma_L$,
  which contradicts the fact that $\gamma_i, \gamma_j \in \Gamma_L$,
  but $\sigma \notin \Gamma_L$. Finally, $\gamma_i\Gamma_F =
  \gamma_j\Gamma_F$ only for $i=j$ by construction. This proves the
  claim.

  Therefore, $N_{F/k}(\ww) = N_{F/L}(\ww)N_{F/L}(\sigma(\ww))$. We note that
  $\sigma$ induces a $k$-automorphism of the variety $R_{F/k}(\Gm{F})$:
  this is clear from the functor-of-points description of $R_{F/k}(\Gm{F})$.

  Using this observation, we see that the variety $Y' \subset \AA^1_k \times
  (R_{F/k}(\Gm{F}))^2$ with coordinates $(t,\ww,\yy)$ defined
  by
  \begin{equation*}
    N_{F/k}(\ww)(t-\sqrt a) = \rho N_{F/L}(\yy)
  \end{equation*}
  (i.e. equation~(\ref{eq:torsor_BHB}) with $K$ replaced by $F$) is isomorphic
  to the product $\T_U \times R_{F/k}(\Gm{F})$ with coordinates $(t,\xx,\yy)$
  subject to~(\ref{eq:torsor_equation_2}). The isomorphism is defined by the
  map
  \begin{equation*}
    (t,\ww,\yy) \mapsto (t, (\ww \sigma(\ww))^{-1} \yy,\ww),
  \end{equation*}
  the inverse substitution being given by
  \begin{equation*}
    (t,\xx,\yy) \mapsto (t, \yy, \xx\yy \sigma(\yy)).
  \end{equation*}
  Since $Y'$ satisfies weak approximation by \cite[Theorem~2]{arXiv:1109.0232}
  and since $R_{F/k}(\Gm{F})$ is rational and therefore has non-trivial
  $k_v$-points for any place $v$, this implies that $\T_U$ satisfies weak
  approximation.
\end{proof}

\begin{proof}[Proof of Theorem~\ref{thm:Q}]
  If $P(t)$ is split over $\QQ$ with two distinct roots, then
  Theorem~\ref{thm:Q} is a special case of
  \cite[Theorem~1.1]{MR1961196}. If it is split over $\QQ$ with one
  double root, $U \subset X$ as in Proposition~\ref{prop:torsor} is a
  principal homogeneous space of a torus, and Theorem~\ref{thm:Q}
  holds by \cite{MR631309}.

  Next, assume that $P(t)$ is irreducible over $\QQ$. Assume that there is no
  Brauer--Manin obstruction to the Hasse principle on a smooth and proper
  model of $X$. Then Lemma~\ref{lemm:torsorbis} shows that universal
  torsors $\T$ over $X$ exist. By Proposition~\ref{lem:WA_over_Q}, $\T_U$
  satisfies weak approximation. Proposition~\ref{prop:torsor} shows that
  $\kbar[X]^\times = \kbar^\times$ and that $\Pic(\Xbar)$ is free of finite
  rank. Then an application of Proposition~\ref{lem:torsor} gives the result.
\end{proof}

\begin{cor}\label{cor:non-split}
  If the quadratic polynomial $P(t) \in \QQ[t]$ is \emph{not} split in the
  Galois closure of $K/\QQ$, then the Hasse principle and weak
  approximation hold on any smooth proper model of $X \subset
  \AA^{n+1}_\QQ$ defined by~(\ref{eq:variety}).
\end{cor}

\begin{proof}
  By \cite[Theorem~2.2]{Wei11}, the smooth proper model $X^c$
  satisfies $\Br(X^c)=\Br_0(X^c)$, so the result follows immediately from
  Theorem~\ref{thm:Q}.
\end{proof}

Finally, we generalize Theorem~\ref{thm:Q} to equations involving a
multivariate polynomial $P(t_1,\dots,t_\ell)$, using techniques
developed by Harari in \cite{MR1478028}:

\begin{cor}
  Let $P_0, P_1, P_2$ be polynomials in $\ell-1$ variables $t_2,
  \dots, t_\ell$ over $\QQ$ of arbitrary degree satisfying
  \begin{equation*}
    \gcd(P_0(t_2, \dots, t_\ell), P_1(t_2, \dots, t_\ell), P_2(t_2, \dots, t_\ell))=1.
  \end{equation*}

  Let $K$ be an arbitrary number field of degree $n = [K : \QQ]$. Then the
  Brauer--Manin obstruction to the Hasse principle and weak approximation is
  the only obstruction on any smooth proper model of $X \subset
  \AA^{n+\ell}_\QQ$ defined by the equation
  \begin{equation*}
    t_1^2 \cdot P_2(t_2, \dots, t_\ell) + t_1\cdot P_1(t_2, \dots, t_\ell) +
    P_0(t_2, \dots, t_\ell) = N_{K/\QQ}(\zz).
  \end{equation*}
\end{cor}

\begin{proof}
  Consider the projection $\pi: X \to \AA^{\ell-1}_\QQ$ defined by
  $(\tt,\zz) \mapsto (t_2, \dots, t_\ell)$ and consider the closed subset
  \begin{equation*}
    F = \{P_0(t_2, \dots, t_\ell) = P_1(t_2, \dots, t_\ell) = P_2(t_2, \dots, t_\ell)=0\}
  \end{equation*}
  of $\AA^{\ell-1}_\QQ$, which is of codimension at least $2$ by assumption.

  The fibers of $\pi$ over $\AA^{\ell-1}_\QQ \setminus F$ are geometrically
  integral. The fiber over each rational point in this set is defined by
  $P(t_1)=N_{K/\QQ}(\zz)$ for some non-zero polynomial $P(t_1)$ of degree at
  most $2$. By Theorem~\ref{thm:Q} for quadratic $P(t_1)$, by rationality for
  linear $P(t_1)$ and by \cite{MR631309} for constant $P(t_1)$, this has the
  property that the Brauer--Manin obstruction to the Hasse principle and weak
  approximation is the only obstruction on any smooth proper model. The
  generic fiber of $\pi$ is a rational variety. Therefore, the result follows
  by an application of \cite[Th\'eor\`eme~3.2.1]{MR1478028} (where we can check
  over $\kbar$ that \emph{condition (*)} of that result holds in our case).
\end{proof}

\section{Brauer groups for quartic norms}\label{sec:brauer}

For $X$ defined by (\ref{eq:variety}) with $K/k$ a quartic extension
of number fields, we show in this section that the Brauer group of a
smooth proper model $X^c$ of $X$ is trivial in certain cases. Then an
application of Theorem~\ref{thm:Q} proves the Hasse principle and weak
approximation on $X$ for $k=\QQ$. In other cases, we show that the Brauer group is
non-trivial. In the next section, we will also give an explicit example which
illustrates the fact that a non-trivial Brauer class can give rise to
an obstruction to weak approximation in this situation.

Recall that the Brauer group is a birational invariant of smooth, proper
varieties: therefore it suffices to study the Brauer group of a given (smooth)
compactification of $X$. Following the ideas developed in
\cite[Section~2]{MR2053456}, we will use a certain partial compactification
$Y$ of $X$ contained in a smooth proper model $X^c$. It suffices to consider
this particular model. Note that the natural maps $\Br(X^c) \to \Br(Y) \to
\Br(X)$ are injective.

A classical argument, based on the Hochschild--Serre spectral sequence (see
\cite[Proposition~2.3]{MR2053456}), gives an exact sequence of the form
\begin{equation}\label{eq:hochschild-serre}
  0 \to \Br_0(Y) \to \Br(Y) \to H^1(k, \Pic(\Ybar)) \to H^3(k,\kbar^\times).
\end{equation}
If $k$ is a number field, then $H^3(k,\kbar^\times) = 0$. Let
$\Gamma_k$ be the abolute Galois group of $k$. Let $L$ be the
splitting field of $P(t)$.

\begin{lemma}\label{lem:exact}
  Consider the norm one torus $T=R_{K/k}^1(\Gm{K})$. Let $T^c$ be a
  smooth compactification of $T$ appearing in the construction of the
  partial compactification $Y$ of $X$ as in
  \cite[Section~2]{MR2053456}. Then
  \begin{equation*}
    0 \to Q_L/\Res_{k/L}(Q_k) \to H^1(k,\Pic(\overline Y)) \to H^1(k,\Pic(\Tcbar))
  \end{equation*}
  is exact, where
  \begin{equation*}
    \begin{aligned}
      Q_k&=\ker(H^1(k,\QQ/\ZZ)\to H^1(K,\QQ/\ZZ)),\\
      Q_L&=\ker(H^1(L,\QQ/\ZZ)\to H^1(K \otimes_k L,\QQ/\ZZ)).
    \end{aligned}
  \end{equation*}
  Here, $\Res_{k/L}:H^1(k,\QQ/\ZZ) \to H^1(L,\QQ/\ZZ)$ and all other maps are restriction maps.
\end{lemma}

\begin{proof}
  Because of \cite[Proposition 2.5]{MR2053456}, we have the exact
  sequence
  \begin{equation}\label{bs}
    0  \rightarrow H^1(k,\widehat T\otimes_\ZZ
    \ZZ[\Gamma_k/\Gamma_L])/j_*H^1(k,\widehat T) \rightarrow H^1(k,\Pic(\overline Y))
    \rightarrow \cyr{X}^2_\omega (\widehat T)_P \rightarrow 0,
  \end{equation}
  where $j: \ZZ \to \ZZ[\Gamma_k/\Gamma_L]$ is defined by $1 \mapsto
  \sum_{\sigma\Gamma_L \in
    \Gamma_k/\Gamma_L}\sigma\Gamma_L$. Furthermore, $\That$ is the
  character group of $T$ and $\cyr{X}^2_\omega (\widehat T)_P$ is a
  subgroup of $\cyr{X}^2_\omega (\widehat T)$
  \cite[D\'efinition~2.4]{MR2053456}.

  By Shapiro's lemma and the long exact sequence in Galois cohomology
  associated to the short exact sequence $0 \to \ZZ \to
  \ZZ[\Gamma_k/\Gamma_L] \to \That \to 0$, the first term
  of~(\ref{bs}) is isomorphic to $Q_L/\Res_{k/L}(Q_k)$ as above.

  Furthermore, we have $\cyr{X}^2_\omega(\widehat{T}) \cong
  H^1(k,\Pic(\overline{T}^c))$ for all smooth compactifications
  $T^c$ of $T$ by \cite[Proposition~9.5]{MR878473}. This gives the result.
\end{proof}

In the following result, a more detailed version of
Theorem~\ref{thm:brauer_short} stated in the introduction, we determine
the Brauer groups.

\begin{theorem}\label{thm:brauer}
  Let $k$ be a field of characteristic zero with $H^3(k,\kbar^\times)
  = 0$. Let $K/k$ be a field extension of degree $4$ with Galois
  closure $K^\gc$. Let $P(t)$ be an irreducible quadratic polynomial
  over $k$ with splitting field $L$. Let $X^c$ be a smooth proper
  model of $X\subset \AA^5_k$ defined by~(\ref{eq:variety}).

 If we are in the case that $K/k$ is not Galois, that $P(t)$ remains irreducible
  over $K$ and that $K \cdot L/k$ is Galois with $\Gal(K\cdot
  L/L)\cong \ZZ/4\ZZ$, then we have $\Br(X^c)/\Br_0(X^c) \cong \ZZ/2\ZZ$. Moreover, a
  representative for the non-trivial element of $\Br(X^c)/\Br_0(X^c)$
  is given by $\Cor_{L/k}(t - \sqrt{a},\chi)$, where $\chi \in
  \Hom(\Gal(K\cdot L/L),\QQ/\ZZ)$ is a character that is non-trivial
  in $Q_L/\Res_{k/L}(Q_k)$.

  Otherwise, $\Br(X^c) = \Br_0(X^c)$.
\end{theorem}

\begin{proof}
  We have the following mutually disjoint cases, where we will first
  show that the Brauer group is trivial in
  cases~\eqref{split}--\eqref{Z2}; then we will show that it is
  non-trivial in case~\eqref{Z4}.
  \begin{enumerate}
  \item \label{split} $P(t)$ is split over $K$, i.e. $L \subset K$,
  \item \label{notsplitoverKc} $P(t)$ is irreducible over $K^\gc$,
  \item \label{S4} $P(t)$ is irreducible over $K$ and split in $K^\gc$, with
    $K^\gc \neq K \cdot L$.
  \item \label{Z2} $P(t)$ is irreducible over $K$ and split in $K^\gc$,
    with $K^\gc=K\cdot L$ and $\Gal(K^\gc/L)\cong \ZZ/2\ZZ \times
    \ZZ/2\ZZ$.
  \item \label{Z4} $P(t)$ is irreducible over $K$ and split in $K^\gc$,
    with $K^\gc=K\cdot L$ and $\Gal(K^\gc/L)\cong \ZZ/4\ZZ$.
  \end{enumerate}

  We can assume without loss of generality that $P(t)=c(t^2-a)$, where
  $a \in k^\times$ is not a square and $c \in k^\times$.

  \medskip

  \textbf{Case~\eqref{split}.} If the extension $K/k$ is Galois in this case, the
  statement follows from results in \cite{Wei11}: we refer to
  \cite[Proposition~1.2(d)]{Wei11} for the case $\Gal(K/k) \cong \ZZ/4\ZZ$ and
  to \cite[Proposition~2.6]{Wei11} for the much harder case $\Gal(K/k) \cong
  \ZZ/2\ZZ \times \ZZ/2\ZZ$.

  If $K/k$ is not Galois in case~\eqref{split}, then since $K \supset
  L=k(\sqrt a)$, there exist $u,v \in k$ such that
  $K=L(\sqrt{u+v\sqrt{a}})$.  The minimal polynomial of
  $\sqrt{u+v\sqrt a}$ over $k$ is $t^4-2ut^2+u^2-av^2$, with roots
  $\pm \sqrt{u\pm v\sqrt a}$ in $K^\gc$. Since $K/k$ is not Galois,
  $\sqrt{u-v\sqrt a}$ is not in $K$ and hence $d=u^2-av^2$ is not a
  square in $k^\times$; in particular, $v \ne 0$. On the other hand,
  $K(\sqrt{d})/k$ is Galois and hence it is the Galois closure of
  $K/k$. Since $\Gal(K(\sqrt{d})/k)$ is a subgroup of $\mathcal{S}_4$
  of order $8$, it must be the $2$-Sylow subgroup $\mathcal{D}_4$ in
  $\mathcal{S}_4$. We now show that $H^1(k,\Pic(\Ybar)) = 0$.

  The group $H^1(k,\Pic(\overline{T}^c))$ (where again $T =
  R^1_{K/k}(\Gm{K})$) is trivial, for example by
  \cite[Proposition~1]{MR687845} (although full details are not given
  there: these can be found in unpublished work of Sansuc, see
  \cite{Sansuc}). Since
  \begin{equation*}
    t^4-2ut^2+u^2-av^2 = (t^2-(u+v\sqrt a))(t^2-(u-v\sqrt a))
  \end{equation*}
  is a factorization into coprime polynomials over $L$,
  we have
  \begin{equation*}
    K \otimes_k L = L[t]/((t^2-(u+v\sqrt a))(t^2-(u-v\sqrt a)))=K
    \oplus K'
  \end{equation*}
  with $K' = L(\sqrt{u-v\sqrt{a}})$. Since $K \ne K'$, we have $K \cap K'=L$. Therefore, in Lemma~\ref{lem:exact},
  \begin{align*}
  Q_L&=\ker(H^1(L,\QQ/\ZZ)\to H^1(K,\QQ/\ZZ)\oplus H^1(K',\QQ/\ZZ))\\
  &=\ker(H^1(L,\QQ/\ZZ)\to H^1(K\cap K',\QQ/\ZZ))=0,
  \end{align*}
   so that $H^1(k,\Pic(\Ybar)) = 0$. Of
  course this implies $\Br(X) = \Br_0(X)$ by the short exact
  sequence~\eqref{eq:hochschild-serre}.

  \medskip

  \noindent\textbf{Case~\eqref{notsplitoverKc}.} In this case, $P(t)$ is irreducible over the Galois
  closure $K^\gc$ of the non-Galois extension $K/k$. The result then follows immediately from \cite[Theorem~2.2]{Wei11}.

  \medskip

  \noindent\textbf{Case~\eqref{S4}.} In this case, we know that $K^\gc
  \ne K\cdot L$ and that $P(t)$ is irreducible over $K$, but split
  over the Galois closure $K^\gc$, so $K\cdot L \subsetneq
  K^\gc$.  Since $[K\cdot L:k]=8$, we have
  $\Gal(K^\gc/k) \cong \mathcal{S}_4$. Again by
  \cite[Proposition~1]{MR687845}, the group
  $H^1(k,\Pic(\overline{T}^c))$ is trivial, so it is enough to prove
  the triviality of
  \begin{equation*}
    Q_L = \ker(H^1(L,\QQ/\ZZ)\to H^1(K \cdot L,\QQ/\ZZ)).
  \end{equation*}
  Non-triviality of this kernel would mean that there exists a
  non-trivial cyclic extension of $L$ contained in $K \cdot L$. But
  $\Gal(K^\gc/L) \cong \mathcal{A}_4$ since it has index $2$ in
  $\Gal(K^\gc/k) \cong \mathcal{S}_4$. Therefore, the Galois group of
  such an extension would be a normal subgroup of index $2$ or $4$ in
  $\mathcal{A}_4$, which does not exist. Therefore $Q_L = 0$ and
  $\Br(X^c) = \Br_0(X^c)$.

  \medskip

  \noindent\textbf{Cases~\eqref{Z2} and \eqref{Z4}.} In these cases, $K^\gc = K
  \cdot L$ has degree $8$ over $k$.  As in the part of
  case~\eqref{split} where $K/k$ is not Galois, we have $\Gal(K^\gc/k)
  \cong \mathcal{D}_4$, with $K/k$ not Galois and $L \subset K^\gc$.

  The group $\mathcal{D}_4$ has five subgroups of order $2$, exactly one of which is normal;
  and it has three normal subgroups of order $4$; moreover, any non-normal subgroup of order $2$
  is contained in exactly one normal subgroup of order $4$, the normal subgroup of order $2$ is contained in  all normal subgroups of order $4$. In the subfield lattice of $K^\gc/k$, we have the
  following intermediate fields, where we
  mark the normal extensions of $k$ by a box, and $K$ is any one of the
  non-normal extensions:
  \newcommand{\normal}[1]{*+<5pt>[F]{#1}}
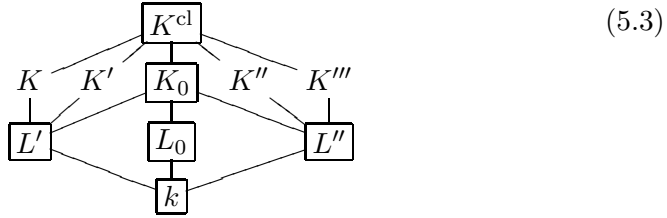
  \begin{equation}\label{eq:diagram}
    \xymatrix@R=0.1in @C=0.1in{&&\normal{K^\gc}\ar@{-}[lld]
      \ar@{-}[ld]\ar@{-}[d]\ar@{-}[rd]\ar@{-}[rrd]\\
      K\ar@{-}[d] & K'\ar@{-}[ld] & \normal{K_0}\ar@{-}[lld]
      \ar@{-}[d]\ar@{-}[rrd] & K''\ar@{-}[rd] & K'''\ar@{-}[d]\\
      \normal{L'} \ar@{-}[rrd]&& \normal{L_0} \ar@{-}[d]
      && \normal{L''}\ar@{-}[lld]\\&& \normal{k}
    }
  \end{equation}
  This diagram shows that $L=L'$ in the non-Galois part of
  case~\eqref{split}, $L=L''$ in case~\eqref{Z2} with $\Gal(K^\gc/L)
  \cong \ZZ/2\ZZ \times \ZZ/2\ZZ$, and $L=L_0$ in case~\eqref{Z4} with
  $\Gal(K^\gc/L) \cong \ZZ/4\ZZ$.  In all three cases, $\Gal(K^\gc/k)
  \cong \mathcal{D}_4$ implies that $H^1(k,\Pic(\overline{T}^c)) = 0$,
  as we have seen in case~\eqref{split}.

  In the two cases $L=L''$ and $L=L_0$, the polynomial $P(t)$ remains
  irreducible in $K$, and hence we have $K \otimes_k L \cong K(\sqrt
  a) = K \cdot L=K^\gc$, so the group $Q_L \cong H^1(K^\gc/L,\QQ/\ZZ)$
  has order $4$.  Furthermore, $Q_k$ has order $2$ because the diagram
  shows that there is only one non-trivial cyclic extension of $k$
  contained in $K$, namely $L'$.  Since $L \cap K = k$, the
  restriction map $\Res_{k/L}$ sends $Q_k$ injectively into $Q_L$, and
  we conclude that $H^1(k, \Pic(\Ybar)) \cong \ZZ/2\ZZ$. Now the short
  exact sequence~(\ref{eq:hochschild-serre}) implies that the quotient
  $\Br(Y^c)/\Br_0(Y^c)$ -- and therefore also $\Br(X^c)/\Br_0(X^c)$ --
  injects into $\ZZ/2\ZZ$.

  By \cite[Remark on p.~76]{MR2053456}, we know that $\Br(Y)/\Br_0(Y)$
  is generated by the element $B = \Cor_{L/k}(t-\sqrt a, \chi) \in
  \Br(k(X))$, where $\chi \in Q_L$ with $\chi \not\in Q_k$.

  \medskip

  \noindent\textbf{Case~\eqref{Z2}.} For $L=L''$, we will show that $B$ is
  ramified at $t=\infty$. This implies that $\Br(X^c)=\Br_0(X^c)$.  In
  case~\eqref{Z4} with $L=L_0$, we will show below that $B$ extends to
  a non-trivial element of $\Br(X^c)$.

  We consider the quadratic extension $(K^\gc)^{\ker(\chi)}$ of $L$
  associated to $\chi$. If this were $K_0$, then this subfield of
  $K^\gc/L$ would come from the subfield $L'$ of $K/k$, so $\chi$
  would be trivial in $Q_L/\Res_{k/L}(Q_k)$, contrary to our choice of
  $\chi$. So we may assume that $(K^\gc)^{\ker(\chi)}=K''$ in
  diagram~(\ref{eq:diagram}).

  The variety $Y_{K''} = Y \times_k K''$ contains
  an open affine $K''$-subvariety $V$ defined by
  \begin{equation*}
    N_{K^\gc/K''}(\zz_1)N_{K^\gc/K''}(\zz_2)=c(t^2-a) \text{ and }t\neq 0.
  \end{equation*}
  Let $W$ be the smooth affine $K''$-variety defined by
  \begin{equation*}
    N_{K/K''}(\zz'_1)N_{K/K''}(\zz'_2)=c(1-at'^2).
  \end{equation*}
  The open subvariety of $W$ defined by $t'\neq 0$ is isomorphic to
  $V$ by the map $(t',\zz_1',\zz_2') \mapsto
  (1/t',\zz_1'/t',\zz_2)$. Let $D$ be the divisor of $W$ defined by
  $t'=0$.  It is easy to see that the divisor $D$ is geometrically
  irreducible. Hence $\kbar \cap K''(D)=K''$, where $K''(D)$ is
  the function field of $D$. The local ring $A_D\subset K''(Y)$
  associated to the divisor $D$ is a discrete valuation ring and
  $\ord_{A_D}(t)=-1$. We have $\kappa_{A_D}=K''(D)$.

  The natural restriction map $\Res_{k/K''}:\Br(Y)\rightarrow
  \Br(Y_{K''})$ factorizes as $\Res_{k/K''} =
  \Res_{L/K''}\circ\Res_{k/L}$. We will exchange restriction and
  corestriction using the double coset formula
  \cite[Proposition~I.1.5.6]{MR2392026}. The coset decomposition is $\Gal(\kbar/k) =
  \Gal(\kbar/L) \cup \sigma\Gal(\kbar/L)$ for some $\sigma \in
  \Gal(\kbar/k)$ that is non-trivial on $L$. Applying
  \cite[Proposition~1.1.3]{MR1285781} in the last step with
  $v_{A_D}(t\pm \sqrt a)=-1$, we have
  \begin{align*}
    \partial_{A_D}(\Res_{k/K''}(B))&=\partial_{A_D}(\Res_{L/K''}(\Res_{k/L}(\Cor_{L/k}(t-\sqrt a, \chi))))\\
    &=\partial_{A_D}((t-\sqrt{a},\Res_{L/K''}(\chi))+(t+\sqrt{a},\Res_{L/K''}(\chi^\sigma)))\\
    &=-\Res_{L/\kappa_{A_D}}(\chi+\chi^\sigma) \in \Hom(\Gal(\kappabar_{A_D}/\kappa_{A_D}),\QQ/\ZZ).
  \end{align*}

  By diagram~(\ref{eq:diagram}), $L'\cdot K''=K^\gc$.  Since $\chi$
  corresponds to $K''$, we know that $\chi^\sigma$ corresponds to its
  conjugate $K'''$, hence $\chi+\chi^\sigma$ corresponds to $K_0$.
  Since
   \begin{equation*}
     \kappa_{A_D}\cap K_0= K''(D)\cap \kbar \cap
     K_0=K''\cap K_0=L,
   \end{equation*}
   we have
   \begin{equation*}
     \Res_{L/\kappa_{A_D}}(\chi+\chi^\sigma) \neq 0
     \in \Hom(\kappabar_{A_D}/\kappa_{A_D},\QQ/\ZZ).
   \end{equation*}
   Therefore, $\Res_{k/K''}(B)\not \in \Br(Y_{K''}^c)$, hence $B\not \in
   \Br(X^c)$.

   \medskip

   \noindent\textbf{Case~\eqref{Z4}.} As discussed above in the
   context of case~\eqref{Z2}, it remains to show that $B=\Cor_{L/k}(t -
   \sqrt{a},\chi) \in \Br(k(X))$ extends to $\Br(X^c)$. We have
   $L=L_0$ in diagram~\eqref{eq:diagram}.

   It is sufficient to show that for any discrete valuation ring $A$
   of $k(X)$ that corresponds to a valuation that is trivial on $k$,
   with residue field $\kappa_A$, the residue map \cite[\S
   1.1]{MR1285781} applied to $B$ gives the zero map
  \begin{equation*}
    \partial_A(B) \in H^1(\kappa_A, \QQ/\ZZ) =
    \Hom(\Gal(\kappabar_A/\kappa_A), \QQ/\ZZ).
  \end{equation*}

  Let us therefore focus on proving the the triviality of
  $\partial_A(B)$ for any such discrete valuation ring $A$. Let $g \in
  \Gal(\kappabar_A/\kappa_A)$. We extend the embedding $k \subset
  \kappa_A$ to an embedding $\kbar \subset \kappabar_A$, so that $g$
  acts also on $\kbar$. Let $K^{\gc,g}$ be the subfield of $K^\gc$ fixed
  by $g$, with cyclic Galois group $\Gal(K^\gc/K^{\gc,g})$.

  Since $X$ is geometrically integral, $X \times_k K^\gc$ is an irreducible
  variety with function field $K^\gc(X)$.
  We can extend $A$ to a discrete valuation
  ring $A_{K^\gc}$ of $K^\gc(X)$ with residue field
  $\kappa_{A_{K^\gc}}=\kappa_A \cdot K^\gc$. Indeed, the completion of
  $k(X)$ for the given valuation is isomorphic to $\kappa_A((T))$,
  where $T$ is a uniformizer. The map $k(X)\to \kappa_A((T))$ gives
  the natural map $$K^\gc \otimes_k k(X) \to K^\gc \otimes_k
  \kappa_A((T)).$$ Composing with the natural map $K^\gc\otimes_k \kappa_A((T)) \to (K^\gc \cdot
  \kappa_A)((T))$, we have the map $$K^\gc \otimes_k k(X) \to (K^\gc \cdot
  \kappa_A)((T)).$$
Using
  $K^\gc(X) = K^\gc \otimes_k k(X)$, we get the map $$K^\gc(X) \to (K^\gc \cdot
  \kappa_A)((T))$$ that is injective since both objects
  are fields. We can see that $(K^\gc \cdot \kappa_A)((T))$ is a discrete
  valuation field and the valuation is given by the uniformizer $T$;
  the valuation restricted to $K^\gc(X)$ induces a discrete valuation of
  $K^\gc(X)$, and we denote the respective discrete valuation ring by
  $A_{K^\gc}$ and the residue field by $\kappa_{A_{K^\gc}}$.
  We have the induced map $A_{K^\gc}\to (K^\gc \cdot
  \kappa_A)[[T]]$. Then we have the injection $\kappa_{A_{K^\gc}}\to
  K^\gc \cdot \kappa_A$, since the residue field of $(K^\gc \cdot
  \kappa_A)[[T]]$ is $K^\gc \cdot \kappa_A$.  Since $k(X)$ is dense in
  $\kappa_A((T))$, we conclude that $K^\gc(X)$ is dense in $(K^\gc
  \cdot \kappa_A)((T))$, which implies that the map
  $\kappa_{A_{K^\gc}}\to K^\gc \cdot \kappa_A$ is
  surjective. Therefore the map $\kappa_{A_{K^\gc}}\to K^\gc \cdot
  \kappa_A$ is an isomorphism.

  For any intermediate field $M$ of $K^\gc/k$, we have similarly the
  valuation ring $A_M$ of $M(X)$ with residue field $\kappa_{A_M}$. We
  write $A_g$ for $A_{K^{\gc,g}}$.

  To show that $\partial_A(B)(g)$ is zero for any $g$, we distinguish
  two cases:
  \begin{enumerate}[(\ref{Z4}.i)]
  \item\label{it:easy} $L \not\subset K^{\gc,g}$,
  \item\label{it:hard} $L \subset K^{\gc,g}$.
  \end{enumerate}

  \medskip

  \noindent\textbf{Case~(\ref{Z4}.\ref{it:easy}).} We assume that $L
  \not\subset K^{\gc,g}$. We note that $\Gal(K^\gc/K^{\gc,g})$ is a
  cyclic subgroup of $\mathcal{D}_4$. We see in
  diagram~(\ref{eq:diagram}) that the only such subfields of $K^\gc$
  not containing $L=L_0$ are $K,K',K'',K'''$. In any case, $L\cdot
  K^{\gc,g} = K^\gc$.

  By \cite[Proposition~1.1.1]{MR1285781}, we have the commutative
  diagram
  \begin{equation*}
    \begin{CD}
      \Br(k(X)) @>\partial_A>> H^1(\kappa_A,\QQ/\ZZ)\\
      @V\Res_{k/K^{\gc,g}}VV @VV\Res_{\kappa_A/\kappa_{A_g}}V\\
      \Br(K^{\gc,g}(X)) @>\partial_{A_g}>> H^1(\kappa_{A_g},\QQ/\ZZ).
    \end{CD}
  \end{equation*}
  Since $\kappa_{A_g} = \kappa_A \cdot K^{\gc,g}$, we have $g \in
  \Gal(\kappabar_A/\kappa_{A_g})$. Hence
  \begin{equation*}
    \partial_A(B)(g) = \Res_{\kappa_A/\kappa_{A_g}}(\partial_A(B))(g)
    = \partial_{A_g}(\Res_{k/K^{\gc,g}}(B))(g).
  \end{equation*}

  By definition, we have $\Res_{k/K^{\gc,g}}(B) =
  \Res_{k/K^{\gc,g}}(\Cor_{L/k}(t-\sqrt a,\chi))$.  Since $L \cap K^{\gc,g}
  = k$, we have $\Gal(\kbar/k) = \Gal(\kbar/L)\cdot
  \Gal(\kbar/K^{\gc,g})$. Exchanging restriction and corestriction
  using \cite[Proposition~I.1.5.6]{MR2392026}
  gives
  \begin{equation*}
    \Res_{k/K^{\gc,g}}(B) =
    \Cor_{L\cdot K^{\gc,g}/K^{\gc,g}}(\Res_{L/L\cdot K^{\gc,g}}(t-\sqrt a, \chi)).
  \end{equation*}
  Since $L \cdot K^{\gc,g}=K^\gc$, the projection formula for cup products
  gives
  \begin{equation*}
    \Res_{L/L \cdot K^{\gc,g}}(t-\sqrt a, \chi) = (t-\sqrt a,\Res_{L/K^\gc}(\chi)).
  \end{equation*}
  But $\Res_{L/K^\gc}(\chi) = 0$ by definition of $Q_L$, hence
  $\Res_{k/K^{\gc,g}}(B) = 0$.

  \medskip

  \noindent\textbf{Case~(\ref{Z4}.\ref{it:hard}).} We assume that $L \subset
  K^{\gc,g}$. We work on $X_L = X \times_k L$. There are maps
  $X_{K^\gc} \to X_L \to X$. Let $A_L = A_g \cap L(X)$ be the
  corresponding discrete valuation ring on $L(X)$. Using the diagram
  \begin{equation*}
    \begin{CD}
      \Br(k(X)) @>\partial_A>> H^1(\kappa_A,\QQ/\ZZ)\\
      @V\Res_{k/L}VV @VV\Res_{\kappa_A/\kappa_{A_L}}V\\
      \Br(L(X)) @>\partial_{A_L}>> H^1(\kappa_{A_L}, \QQ/\ZZ)\\
      @V\Res_{L/K^{\gc,g}}VV @VV\Res_{\kappa_{A_L}/\kappa_{A_g}}V\\
      \Br(K^{\gc,g}(X)) @>\partial_{A_g}>> H^1(\kappa_{A_g}, \QQ/\ZZ)
    \end{CD}
  \end{equation*}
  we get the equality
  \begin{equation*}
    \partial_{A_g}(\Res_{k/K^{\gc,g}}(B)) =
    \Res_{\kappa_{A_L}/\kappa_{A_g}}(\partial_{A_L}(\Res_{k/L}(B))).
  \end{equation*}
  Thus it suffices to show that $\partial_{A_L}(\Res_{k/L}(B)) = 0$ (independent of $g$).

  We apply the double coset formula again: in this case, the coset decomposition
  reduces to $\Gal(\kbar/k) = \Gal(\kbar/L) \cup \sigma\Gal(\kbar/L)$, for any
  $\sigma \in \Gal(\kbar/k)$ that is non-trivial on $L$. Therefore,
  \begin{equation*}
    \Res_{k/L}(B) = (t-\sqrt a, \chi) + (t+\sqrt a, \chi^\sigma),
  \end{equation*}
  where $\chi^\sigma$ is obtained from $\chi$ via conjugation with
  $\sigma$.  We note that $X_L$ is given by $c(t-\sqrt a)(t+\sqrt a) =
  N_{K^\gc/L}(\zz)$ since $K^\gc = K \otimes_k L$. By
  \cite[Proposition~1.1.3]{MR1285781}, we have
  \begin{equation}\label{eq:residue}
    \partial_{A_L}(\Res_{k/L}(B)) =
    v_{A_L}(t-\sqrt a)\Res_{L/\kappa_{A_L}}(\chi)
    +v_{A_L}(t+\sqrt a) \Res_{L/\kappa_{A_L}}(\chi^{\sigma}).
  \end{equation}
  If $v_{A_L}(t-\sqrt a)=v_{A_L}(t+\sqrt a)=0$, then~(\ref{eq:residue})
  vanishes.

  If $v_{A_L}(t-\sqrt a) >0$, then $v_{A_L}(t+\sqrt a) = 0$ since $v_{A_L}(2\sqrt
  a)=0$. Therefore,
  \begin{equation*}
    v_{A_L}(t-\sqrt a) = v_{A_L}(t-\sqrt a)+v_{A_L}(t+\sqrt a) =
    v_{A_L}(N_{K^\gc/L}(\zz))-v_{A_L}(c).
  \end{equation*}
  Now $v_{A_L}(c) = 0$. Let $M = \kappa_{A_L}\cap K^\gc$.
  Then
   \begin{equation*}
    K^\gc \otimes_L M \cong \bigoplus_{i=1}^{[M:L]} K^\gc
  \end{equation*}
  since $K^\gc/L$ is Galois and $M$ is totally split in $K^\gc$.
   Since $A_L$ is a discrete valuation ring of $L(X)$, we have $\widehat{L(X)}\cong \kappa_{A_L}((T))$, where $T$ is a uniformizer.  Hence $$\widehat{L(X)} \cap K^\gc=\kappa_{A_L}\cap K^\gc=M.$$
   We have
  \begin{equation*}
    K^\gc \otimes_L \widehat{L(X)} \cong K^\gc \otimes_L M \otimes_M \widehat{L(X)}
    \cong \bigoplus_{i=1}^{[M:L]} K^\gc \otimes_M \widehat{L(X)}.
  \end{equation*} Therefore,
  \begin{align*}
    v_{A_L}(N_{K^\gc/L}(\zz)) &= v_{A_L}(N_{K^\gc \otimes_L \widehat{L(X)}/\widehat{L(X)}}(\zz))\\
    &=\sum_{i=1}^{[M:L]} v_{A_L}(N_{K^\gc \otimes_M \widehat{L(X)}/\widehat{L(X)}}(\zz_i))\\
    &=\sum_{i=1}^{[M:L]} [K^\gc\otimes_M \widehat{L(X)} : \widehat{L(X)}] \cdot v_{A_{K^\gc}}(\zz_i).
  \end{align*}
  This is a multiple of $[K^\gc:M]$.  Indeed, since $\widehat{L(X)}
  \cap K^\gc =M$, we know that $K^\gc\otimes_M \widehat{L(X)}$ is a
  unramified field extension over $\widehat{L(X)}$ of degree
  \begin{equation*}
    [K^\gc\otimes_M \widehat{L(X)} : \widehat{L(X)}] = [K^\gc:M],
  \end{equation*}
  and $v_{A_{K^\gc}}$ has values in $\ZZ$.  Since the order of
  $\Res_{L/\kappa_{A_L}}(\chi)$ divides $[K^\gc:M]$,
  (\ref{eq:residue}) vanishes. The case $v_{A_L}(t+\sqrt a) >0$ is
  analogous.

  If $v_{A_L}(t-\sqrt a) < 0$, then $v_{A_L}(t+\sqrt a) = v_{A_L}(t -
  \sqrt{a})$ since $v_{A_L}(2\sqrt{a}) = 0$. Since $\Gal(K^\gc/L) \cong
  \ZZ/4\ZZ$, it is easy to see that $\chi^\sigma =
  -\chi$. Hence~(\ref{eq:residue}) vanishes.  This completes the proof
  of $\partial_{A_L}(\Res_{k/L}(B))=0$ in case~(\ref{Z4}.\ref{it:hard}), and
  therefore the proof of the non-triviality of $\Br(X^c)$ in
  case~\eqref{Z4}.
\end{proof}

\begin{cor}\label{cor:quartic}
  Let $K/k$ be a quartic extension of number fields with $k =
  \QQ$. Let $P(t) \in \QQ[t]$ be an irreducible quadratic polynomial
  with splitting field $L$.

  If we are not in the case where $K/k$ is not Galois, $P(t)$ is
  irreducible over $K$ and $K \cdot L / k$ is Galois with $\Gal(K\cdot
  L/L) \cong \ZZ/4\ZZ$, then the Hasse principle and weak
  approximation hold on $X \subset \AA^5_\QQ$ defined by equation
  (\ref{eq:variety})
\end{cor}

\begin{remark}
  J.-L. Colliot-Th\'el\`ene suggested an alternative method to prove
  Theorem~\ref{thm:brauer} in case (\ref{split}), which we will sketch
  here. The proof of Theorem~\ref{thm:degree_4} shows that $X^c$ is
  birational to a smooth, projective $k$-variety $U^c$ equipped with a
  morphism $p: U^c \to S^c$ to a smooth, projective conic $S^c$ over $k$, in
  such a way that the generic fibre of $p$ is a smooth quadric
  $U^c_{k(S)}$ in $\mathbb{P}^4_{k(S)}$. One can then check that
  $\Br(U^c_{k(S)}) \cong \Br(k(S))$ and that $\Br(k)$ surjects onto
  $\Br(S^c)$. An explicit calculation with residues of elements of
  $\Br(U^c_{k(S)})$ then allows one to conclude that $\Br(k)$ surjects
  onto $\Br(U^c)$.
\end{remark}

\section{Failure of weak approximation}\label{sec:counterexample}

Finally, we give a concrete example of a Brauer-Manin obstruction to
weak approximation caused by the non-trivial element of
$\Br(X^c)/\Br_0(X^c)$ described in Theorem~\ref{thm:brauer}.

\begin{example}
  Let $K = \QQ(\sqrt[4]{17})$ and let $P(t) = t^2+1$. Let $X \subset
  \AA^5_\QQ$ be defined by $t^2+1 = N_{K/\QQ}(\zz)$. Then weak
  approximation does not hold on~$X$.

Indeed, consider the adelic point $(\xx_v) \in \prod_v X(k_v)$ given by
  \begin{equation*}
    \xx_v = (t_v, \zz_v) =
    \begin{cases}
      (0,1), & v \ne 17,\\
      (7,\zz_{17}), & v = 17,
    \end{cases}
  \end{equation*}
  where $\zz_{17} \in \QQ_{17}(\sqrt[4]{17})$ is a solution of
  $7^2+1=N_{\QQ_{17}(\sqrt[4]{17})/\QQ_{17}}(\zz_{17})$. Such a
  solution exists: $50$ is even a fourth power in $\QQ_{17}$ by
  Hensel's lemma, since $50 \equiv 2^4 \pmod{17}$.

  Let $\chi$ be a primitive character of the cyclic group
  $\Gal(K(i)/\QQ(i)) \cong \ZZ/4\ZZ$, regarded as element of
  $H^2(\QQ(i),\ZZ)$. Let $B = \Cor_{\QQ(i)/\QQ}(t - i,\chi)$. By
  Theorem~\ref{thm:brauer}, this defines a non-trivial element in
  $\Br(X^c)$. Since
  \begin{equation*}
    \sum_{v \in \Omega_\QQ} \inv_v(B((0,1))) = 0 \in \QQ/\ZZ,
  \end{equation*}
  it suffices to check that $\inv_{17}(B(\xx_{17})) \neq
  \inv_{17}(B((0,1)))$. To do this concrete calculation, one needs to
  fix an embedding of $\QQ(i)$ into $\QQ_{17}$.  Let $\alpha \in
  \mathbb{Q}_{17}$ be such that $\alpha^2 + 1 = 0$ and $\alpha \equiv
  4 \pmod{17}$. Consider the embedding $\iota_\alpha: \QQ(i) \to
  \QQ_{17}$ given by $i \mapsto \alpha$. By functoriality, this gives
  an element $\chi_{\alpha} \in H^2(\QQ_{17},\ZZ)$. The cup products
  \begin{equation*}
 \QQ(i)^* \times H^2(\QQ(i),\ZZ) \to \Br(\QQ(i))
\end{equation*}
and
\begin{equation*}
\QQ_{17}^* \times H^2(\QQ_{17},\ZZ) \to \Br(\QQ_{17})
\end{equation*}
are compatible with the functorial maps induced by $\iota_\alpha$. If
$\chi_{-\alpha}$ is the image of $\chi$ via the embedding $\QQ(i) \to
\QQ_{17}$ defined by $i \mapsto -\alpha$, then one easily checks that
$\chi_{-\alpha} = -\chi_{\alpha} \in H^2(\QQ_{17},\ZZ)$. Hence we get
the equalities
\begin{equation*}
\begin{aligned}
  &B(\xx_{17}) - B((0,1))\\
  &= \Res_{\QQ/\QQ_{17}}(\Cor_{\QQ(i)/\QQ}(t-i,\chi))(\xx_{17}) -
  \Res_{\QQ/\QQ_{17}}(\Cor_{\QQ(i)/\QQ}(t - i,\chi))((0,1))\\
  & = (7 - \alpha,\chi_\alpha) + (7 + \alpha,-\chi_\alpha) -
  (-\alpha,\chi_\alpha) - (\alpha,-\chi_\alpha) \\ & = (7 -
  \alpha,\chi_\alpha) + ((7 + \alpha)^{-1},\chi_\alpha) -
  (-\alpha,\chi_\alpha) - (\alpha^{-1},\chi_\alpha) \\ & =
  \left(\frac{7 - \alpha}{7 +
      \alpha},\chi_\alpha\right)-(-1,\chi_\alpha) \\ & = \left((7 -
    \alpha)^2,\chi_\alpha\right)
\end{aligned}
\end{equation*}
in which we used the double coset formula and the bilinearity of the
cup product, together with the fact that $(7 - \alpha)(7 + \alpha) =
50$ and $-1$ are fourth powers in $\QQ_{17}$. Now we only need to
check that $(7 - \alpha)^2$ is not a norm for the extension
$\QQ_{17}(\sqrt[4]{17})/\QQ_{17}$, but this is clear since $(7 - 4)^2
= 9$ is not a fourth power modulo $17$.
\end{example}


\begin{thebibliography}{CTSanSD87b}

\bibitem[BHB12]{arXiv:1109.0232}
T.~D. Browning and D.~R. Heath-Brown.
\newblock Quadratic polynomials represented by norm forms.
\newblock {\em Geom. Funct. Anal.}, 22(5):1124--1190, 2012.

\bibitem[Bou58]{MR0098114}
N.~Bourbaki.
\newblock {\em \'{E}l\'ements de math\'ematique. 23. {P}remi\`ere partie: {L}es
  structures fondamentales de l'analyse. {L}ivre {II}: {A}lg\`ebre. {C}hapitre
  8: {M}odules et anneaux semi-simples}.
\newblock Actualit\'es Sci. Ind. no. 1261. Hermann, Paris, 1958.

\bibitem[CT03]{MR2011747}
J.-L. Colliot-Th{\'e}l{\`e}ne.
\newblock Points rationnels sur les fibrations.
\newblock In {\em Higher dimensional varieties and rational points ({B}udapest,
  2001)}, volume~12 of {\em Bolyai Soc. Math. Stud.}, pages 171--221. Springer,
  Berlin, 2003.

\bibitem[CTHarSko03]{MR2053456}
J.-L. Colliot-Th{\'e}l{\`e}ne, D.~Harari, and A.~N. Skorobogatov.
\newblock Valeurs d'un polyn\^ome \`a une variable repr\'esent\'ees par une
  norme.
\newblock In {\em Number theory and algebraic geometry}, volume 303 of {\em
  London Math. Soc. Lecture Note Ser.}, pages 69--89. Cambridge Univ. Press,
  Cambridge, 2003.

\bibitem[CTSal89]{MR988101}
J.-L. Colliot-Th{\'e}l{\`e}ne and P.~Salberger.
\newblock Arithmetic on some singular cubic hypersurfaces.
\newblock {\em Proc. London Math. Soc. (3)}, 58(3):519--549, 1989.

\bibitem[CTSan87a]{MR89f:11082}
J.-L. Colliot-Th{\'e}l{\`e}ne and J.-J. Sansuc.
\newblock La descente sur les vari\'et\'es rationnelles. {II}.
\newblock {\em Duke Math. J.}, 54(2):375--492, 1987.

\bibitem[CTSan87b]{MR878473}
J.-L. Colliot-Th{\'e}l{\`e}ne and J.-J. Sansuc.
\newblock Principal homogeneous spaces under flasque tori: applications.
\newblock {\em J. Algebra}, 106(1):148--205, 1987.

\bibitem[CTSanSD87a]{MR870307}
J.-L. Colliot-Th{\'e}l{\`e}ne, J.-J. Sansuc, and P.~Swinnerton-Dyer.
\newblock Intersections of two quadrics and {C}h\^atelet surfaces. {I}.
\newblock {\em J. reine angew. Math.}, 373:37--107, 1987.

\bibitem[CTSanSD87b]{MR876222}
J.-L. Colliot-Th{\'e}l{\`e}ne, J.-J. Sansuc, and P.~Swinnerton-Dyer.
\newblock Intersections of two quadrics and {C}h\^atelet surfaces. {II}.
\newblock {\em J. reine angew. Math.}, 374:72--168, 1987.

\bibitem[CTSD94]{MR1285781}
J.-L. Colliot-Th{\'e}l{\`e}ne and P.~Swinnerton-Dyer.
\newblock Hasse principle and weak approximation for pencils of
  {S}everi-{B}rauer and similar varieties.
\newblock {\em J. reine angew. Math.}, 453:49--112, 1994.

\bibitem[CTSko00]{MR1744112}
J.-L. Colliot-Th{\'e}l{\`e}ne and A.~N. Skorobogatov.
\newblock Descent on fibrations over {${\bf P}^1_k$} revisited.
\newblock {\em Math. Proc. Cambridge Philos. Soc.}, 128(3):383--393, 2000.

\bibitem[CTSkoSD98]{MR1603908}
J.-L. Colliot-Th{\'e}l{\`e}ne, A.~N. Skorobogatov, and P.~Swinnerton-Dyer.
\newblock Rational points and zero-cycles on fibred varieties: {S}chinzel's
  hypothesis and {S}alberger's device.
\newblock {\em J. reine angew. Math.}, 495:1--28, 1998.

\bibitem[FI97]{MR1438827}
E.~Fouvry and H.~Iwaniec.
\newblock Gaussian primes.
\newblock {\em Acta Arith.}, 79(3):249--287, 1997.

\bibitem[Har97]{MR1478028}
D.~Harari.
\newblock Fl\`eches de sp\'ecialisations en cohomologie \'etale et applications
  arithm\'etiques.
\newblock {\em Bull. Soc. Math. France}, 125(2):143--166, 1997.

\bibitem[Has30]{56.0165.03}
H.~Hasse.
\newblock {Die Normenresttheorie relativ-Abelscher Zahlk\"orper als
  Klas\-senk\"orpertheorie im Kleinen.}
\newblock {\em J. reine angew. Math.}, 162:145--154, 1930.

\bibitem[HBSko02]{MR1961196}
D.~R. Heath-Brown and A.~Skorobogatov.
\newblock Rational solutions of certain equations involving norms.
\newblock {\em Acta Math.}, 189(2):161--177, 2002.

\bibitem[Kun82]{MR687845}
B.~{\`E}. Kunyavski{\u\i}.
\newblock Arithmetic properties of three-dimensional algebraic tori.
\newblock {\em Zap. Nauchn. Sem. Leningrad. Otdel. Mat. Inst. Steklov. (LOMI)},
  116:102--107, 163, 1982.
\newblock Integral lattices and finite linear groups.

\bibitem[NSW08]{MR2392026}
J.~Neukirch, A.~Schmidt, and K.~Wingberg.
\newblock {\em Cohomology of number fields}, volume 323 of {\em Grundlehren der
  Mathematischen Wissenschaften [Fundamental Principles of Mathematical
  Sciences]}.
\newblock Springer-Verlag, Berlin, second edition, 2008.

\bibitem[San81a]{MR631309}
J.-J. Sansuc.
\newblock Groupe de {B}rauer et arithm\'etique des groupes alg\'ebriques
  lin\'eaires sur un corps de nombres.
\newblock {\em J. reine angew. Math.}, 327:12--80, 1981.

\bibitem[San81b]{Sansuc}
J.-J. Sansuc.
\newblock {Le principe de Hasse normique dans le cas di\'edral non galoisien},
  unpublished, 1981.

\bibitem[SJ13]{arXiv:1111.4089}
M.~Swarbrick~Jones.
\newblock {A Note On a Theorem of Heath-Brown and Skorobogatov}.
\newblock {\em Q. J. Math.}, 64(4):1239--1251, 2013.

\bibitem[Sko99]{MR1666779}
A.~N. Skorobogatov.
\newblock Beyond the {M}anin obstruction.
\newblock {\em Invent. Math.}, 135(2):399--424, 1999.

\bibitem[Sko01]{MR1845760}
A.~N. Skorobogatov.
\newblock {\em Torsors and rational points}, volume 144 of {\em Cambridge
  Tracts in Mathematics}.
\newblock Cambridge University Press, Cambridge, 2001.

\bibitem[SSko14]{arXiv:1210.5727}
D.~Schindler and A.~N. Skorobogatov.
\newblock Norms as products of linear polynomials.
\newblock {\em J. Lond. Math. Soc. (2)}, 89(2):559--580, 2014.

\bibitem[Wei14]{Wei11}
D.~Wei.
\newblock {On the equation $N_{K/k}(\Xi)=P(t)$}.
\newblock {\em Proc. Lond. Math. Soc. (3), to appear}, arXiv:1202.4115v3, 2014.

\end{thebibliography}
\end{document}